\documentclass[preprint]{elsarticle}
\usepackage{amssymb}
\usepackage{amsmath}
\usepackage{amsfonts}
\usepackage{color}
\usepackage{ulem}
\usepackage{mleftright}

\makeatletter
\def\ps@pprintTitle{%
	\let\@oddhead\@empty
	\let\@evenhead\@empty
	\let\@oddfoot\@empty
	\let\@evenfoot\@empty
}
\makeatother

\newtheorem{theorem}{Theorem}[section]

\newtheorem{example}[theorem]{Example}

\newtheorem{lemma}[theorem]{Lemma}

\numberwithin{equation}{section}
\newenvironment{proof}[1][Proof]{\noindent \textbf{#1} }{\  \rule{0.5em}{0.5em}}
\definecolor{c100}{rgb}{0.5,0,.1}

\journal{...}

\begin{document}

\begin{frontmatter}
		
		\title{Korovkin type theorems for operators acting on functions of polynomial and exponential growth on $[0,\infty)$
		}

		\author[1]{Ulrich Abel} 
		\author[2]{Ana Maria Acu}
		\author[3]{Margareta Heilmann}
		\author[4]{Ioan Ra\c sa}

		\vspace{10cm}
		
		% \author[]{\hspace{10.5cm} }

		\address[1] {Technische Hochschule Mittelhessen, Department Mathematik, Naturwissenschaften und Datenverarbeitung, Friedberg, Germany, 
		e-mail: ulrich.abel@mnd.thm.de }
	\address[2]{Lucian Blaga University of Sibiu, Department of Mathematics and Informatics, Sibiu,  Romania, e-mail: anamaria.acu@ulbsibiu.ro}
	\address[3]{University of Wuppertal, School of Mathematics and Natural Sciences, Wuppertal,  Germany, e-mail: heilmann@math.uni-wuppertal.de}
	\address[4]{Technical University of Cluj-Napoca, Department of Mathematics, Cluj-Napoca, Romania,
		e-mail:  ioan.rasa@math.utcluj.ro }

		\begin{abstract} 	
			{				We establish two Korovkin-type approximation theorems for sequences of
				positive linear operators acting on continuous functions on $[0,\infty)$.
				The first result shows that pointwise convergence on all polynomial test
				functions implies pointwise convergence for every function in the domain
				that is dominated by a polynomial. The second result provides an analogous
				criterion for functions of exponential growth, where the test functions are
				the exponentials $e^{\lambda x}$, $\lambda\in\mathbb{R}$. The results are
				applied to several classical and modified positive linear operators,
				including Baskakov, Sz\'asz--Mirakjan, Phillips, Post--Widder, Gamma, and
				Durrmeyer-type operators. Further applications are given to operators
				associated with modified Bessel functions and hyperbolic functions.
				The approach provides a unified framework for deriving
				pointwise convergence on unbounded intervals from convergence on suitable
				test functions.			
			} 
		\end{abstract}
		
		
	\end{frontmatter}

%	\fbox{\fbox{{\large \jobname.tex}}}

%	\hskip8.5cm ~ {\large \today}

\section{Introduction}

Positive linear operators play an important role in approximation theory,
both from the theoretical point of view and in a wide range of applications.
A fundamental problem in this area is to determine conditions under which a
sequence of positive linear operators converges to the identity operator on
a prescribed class of functions. On compact intervals, such questions are
closely related to classical Korovkin-type approximation results. On
unbounded intervals, however, the situation is more delicate, since the
growth of the functions at infinity has to be taken into account and the
domain of the operators becomes an essential part of the approximation
problem.

Let
$
L_n:D\to C[0,\infty), \,\, n\in\mathbb N,
$
be a sequence of positive linear operators, where $D$ is a linear subspace
of $C[0,\infty)$.  For operators $L_n$ as above, many results are known, guaranteeing a certain
type of convergence of $L_n f$ toward $f$, for $f$ in suitably restricted
subspaces of $D$, e.g.,  the subspace of bounded functions. In this paper we give sufficient conditions entailing the
pointwise convergence of $L_n f$ toward $f$, for \textit{all} functions $f
\in D$. 

Depending on the operator under consideration, the space
$D$ may contain functions of polynomial growth, functions of exponential
growth, or more general classes of continuous functions. Typical examples
are provided by the classical Baskakov and Sz\'asz--Mirakjan operators, as
well as by their Durrmeyer-type modifications and by Gamma and Post--Widder
type operators.

For approximation processes on $[0,\infty)$, polynomial test functions are
particularly natural. If $\Pi$ denotes the space of all polynomial
functions, the convergence
$
L_np\to p,\,\, p\in\Pi,
$
contains information about all moments of the operators. In many important
examples, these moments can be computed explicitly, and their asymptotic
behavior can therefore be used to establish convergence on $\Pi$. This is
the case, for instance, for the Baskakov and Sz\'asz--Mirakjan operators
considered in \cite{Ismail},  and for several Gamma and Post--Widder type operators
preserving prescribed monomials \cite{Dolomites}.

A natural question is whether convergence on all polynomials can be extended
to a larger class of functions without imposing additional smoothness
assumptions. In the first part of the present paper we show that this is
indeed possible for continuous functions whose absolute value is dominated
by a polynomial. More precisely, we prove that if $D$ contains $\Pi$, if
$L_n$ are positive linear operators, and if
$
L_np\to p
$
pointwise on $[0,\infty)$ for every $p\in\Pi$, then the same pointwise
convergence holds for every $f\in D$ satisfying
$
|f(x)|\leq p(x),\,\, x\geq0,
$
for some polynomial $p\in\Pi$.

The applicability of this result is illustrated by several examples. Besides
the classical Baskakov and Sz\'asz--Mirakjan operators, we consider the
Phillips operators. We also discuss Post--Widder and Gamma type operators fixing
the monomials $e_0$ and $e_j$, as well as related modified operators studied
in \cite{Dolomites}. For these families, the moments are explicitly known,
and their asymptotic behavior immediately gives convergence on polynomial
functions. The general theorem then extends this convergence to all
continuous functions in the corresponding domains that are dominated by a
polynomial. 

A second class of approximation processes on the positive half-line is
naturally associated with functions of exponential growth. Such spaces arise,
for example, in connection with Sz\'asz--Mirakjan type operators and with
operators acting in exponential weight spaces; see, among others,
\cite{Herzog1,LR}. In this setting  the exponential functions
$
\exp_\lambda(x):=e^{\lambda x}$, $\lambda\in\mathbb R
$, $x\geq 0$,
form a natural test family.

The second main result of the paper provides an exponential analogue of the
polynomial convergence theorem. We prove that if $D$ contains
$\exp_\lambda$ for every $\lambda\in\mathbb R$, and if
$
L_n\exp_\lambda\to\exp_\lambda
$
pointwise on $[0,\infty)$ for every $\lambda\in\mathbb R$, then
$
L_nf\to f
$
pointwise for every $f\in D$ satisfying
$
|f(x)|\leq A e^{\omega x},\,\, x\geq0,
$
for some $A>0$ and $\omega>0$.

To illustrate this result, we prove the pointwise
convergence of the classical Sz\'asz--Mirakjan operators for continuous
functions of at most exponential growth. We then apply the theorem to two
further families of positive linear operators. 

Thus, the two convergence principles established in this paper have a common
structure. In the polynomial case, convergence on $\Pi$ is transferred to
continuous functions dominated by polynomials, whereas in the exponential
case, convergence on the family of exponential functions is transferred to
continuous functions dominated by an exponential function. The common
ingredient is positivity of the operators, which allows suitable majorants and minorants to
convert convergence on a distinguished test family into convergence on a
substantially larger class of functions.

In this paper we are concerned with limits 
$n\to\infty$. Therefore, if nothing else is specified, throughout the paper we always consider  $n$ to be sufficiently large depending on the corresponding context.

The paper is organized as follows. In Section~2 we establish the convergence
criterion for functions of polynomial growth and apply it to several
classical and modified positive linear operators, including Baskakov,
Sz\'asz--Mirakjan, Phillips, Gamma, Post--Widder and
Sz\'asz--Mirakjan--Durrmeyer type operators. In Section~3 we prove the
corresponding result for continuous functions of exponential growth and
illustrate it by the classical Sz\'asz--Mirakjan operators, the operators
$A_n^0$ related to a modified Bessel function, and the operators $R_n$
associated with the hyperbolic cosine.

\section{Convergence for functions with polynomial growth}

\begin{theorem}
	\label{T16}
	Let $D$ be a linear subspace of $C[0,\infty)$ such that $\Pi\subset D$,
	and let $L_n:D\rightarrow C[0,\infty)$ be positive linear operators.
	Suppose that
	\[
	(L_np)(x)\rightarrow p(x),\,\, n\rightarrow\infty,
	\]
	for every $p\in\Pi$, pointwise on $[0,\infty)$. Let $f\in D$ be such that
	there exists $p\in\Pi$ satisfying
	\[
	|f(x)|\leq p(x),\,\, x\geq0.
	\]
	Then $%
	L_{n}f\rightarrow f$ as $n\rightarrow \infty $ pointwise on $[0,\infty )$.
\end{theorem}

For the proof we need the following lemma.

\begin{lemma}
\label{L6.1} Let $f\in C[0,\infty )$, $p\in \Pi $, $p(x)\geq f(x)$, $x\in
\lbrack 0,\infty )$. For each $a\geq 0$ and $\varepsilon >0$, there exists $%
P\in \Pi $ such that 
\begin{equation}
P(a)=f(a)+\varepsilon ,\quad P(x)\geq f(x),\quad x\geq 0.  \label{equ1}
\end{equation}
\end{lemma}

\begin{proof}
First we consider $a>0$. Then, there exists $\ell\in \left(0,a/2\right]$, such that 
\begin{equation}  \label{equ2}
f(a) + \varepsilon \geq f(x), \, \,x \in (a - 2\ell, a + 2\ell).
\end{equation}

Let $m \in \mathbb{N}$, $2m > \deg p$. Then, there exists $M > a+2\ell$ with 
\begin{equation*}
\left( \frac{x - a}{\ell} \right)^{2m} + f(a) + \varepsilon \geq p(x), \quad
x > M,
\end{equation*}
which implies that 
\begin{equation}  \label{equ3}
\left( \frac{x - a}{\ell} \right)^{2m} + f(a) + \varepsilon \geq f(x), \quad
x > M.
\end{equation}

Set $U := [0, a - 2\ell] \cup [a + 2\ell, M]$. Since $|(x-a)/\ell|\geq2$ for $x\in U$, we can choose $K\geq1$ such that 
\begin{equation}  \label{equ4}
\min \left \{ K \left( \frac{x - a}{\ell} \right)^{2m} + f(a) + \varepsilon
\; \middle|\; x \in U \right \} \geq \max \left \{ p(x) \; \middle|\; x \in U \right \}.
\end{equation}

Define 
\begin{equation*}
P(x) := K \left( \frac{x - a}{\ell} \right)^{2m} + f(a) + \varepsilon.
\end{equation*}

Using (\ref{equ2})--(\ref{equ4}) we see that $P$ satisfies (\ref{equ1}). 

For $a=0$, choose $\ell>0$ such that
\[
f(0)+\varepsilon\geq f(x),\,\, x\in[0,2\ell).
\]
The same argument, with
\[
P(x):=K\left(\frac{x}{\ell}\right)^{2m}+f(0)+\varepsilon,
\]
completes the proof.
\end{proof}

\medskip

\begin{proof}[Proof of Theorem~\protect \ref{T16}]
Let $a \geq 0$, $\varepsilon > 0$. Lemma \ref{L6.1} applied to $f$ and $-f$, respectively, implies the existence of 
$P, Q \in \Pi$, such that 
\begin{equation*}
Q(x) \leq f(x) \leq P(x), \quad x \geq 0; \quad Q(a) = f(a) - \varepsilon,\
\ P(a) = f(a) + \varepsilon.
\end{equation*}
Now 
$
(L_n Q)(a) \leq (L_n f)(a) \leq (L_n P)(a)$ and we conclude that 
\begin{eqnarray*}f(a) - \varepsilon &=& Q(a) =
\lim_{n\to \infty} (L_n Q)(a)\\
&\leq& \liminf_{n\to \infty} (L_n f)(a) \leq \limsup_{n\to \infty} (L_n f)(a) \leq
\lim_{n\to \infty} (L_n P)(a)\\
& =& P(a) = f(a) + \varepsilon.
\end{eqnarray*}
Therefore, for each $\varepsilon > 0$ we have 
\begin{equation*}
f(a) - \varepsilon \leq \liminf_{n\to \infty} (L_n f)(a) \leq \limsup_{n\to
\infty} (L_n f)(a) \leq f(a) + \varepsilon.
\end{equation*}
It follows that $(L_n f)(a) \to f(a)$ as $n \to \infty$.
\end{proof}

\begin{example}
	Let $V_{n}$ be the classical Baskakov operators and $S_{n}$ the
	Sz\'{a}sz--Mirakjan operators. Let $f\in C[0,\infty)$ be such that there
	exists $p\in\Pi$ with $|f(x)|\leq p(x)$, $x\geq 0$. Then,
$
	V_{n}f\rightarrow f$
	and $
	S_{n}f\rightarrow f
$
	as $n\rightarrow\infty$, pointwise on $[0,\infty)$.
	
	Indeed, the Baskakov and the Sz\'{a}sz--Mirakjan operators are
	exponential-type operators (see \cite[Theorem~3.3]{Ismail}). From
	\cite[Proposition~2.9]{Ismail}, it follows that
	$
	V_{n}p\rightarrow p$
	and $
	S_{n}p\rightarrow p
$
	as $n\rightarrow\infty$, pointwise on $[0,\infty)$, for every $p\in\Pi$.
	Thus, the assertion follows from Theorem~\ref{T16}.
\end{example}

\begin{example}
	Consider the Phillips operators (see \cite{Phillips})
	\[
	(\widetilde{S}_{n}f)(x)
	:=s_{n,0}(x)f(0)
	+n\sum_{k=1}^{\infty}s_{n,k}(x)
	\int_{0}^{\infty}s_{n,k-1}(t)f(t)dt,
	\]
	where $s_{n,k}(x)=e^{-nx}\dfrac{(nx)^k}{k!}$, $n\in {\mathbb N}$, $k\in  {\mathbb N}_0$, $x\geq 0$.
	
	Their moments are given by
	\[
	(\widetilde{S}_{n}e_{r})(x)
	=\sum_{j=1}^{r}\binom{r-1}{j-1}
	\frac{r!}{j!}n^{j-r}x^{j},
	\qquad r\in\mathbb{N},
	\]
	see \cite[Lemma 2.1]{HR}. Hence,
	\[
	\lim_{n\to\infty}(\widetilde{S}_{n}e_r)(x)
	=x^r=e_r(x), \qquad x\geq0,
	\]
	since all the terms corresponding to $j<r$ tend to zero, while the
	term corresponding to $j=r$ is $x^r$. Moreover,
	$\widetilde{S}_ne_0=e_0$. Thus, by linearity,
	\[
	\lim_{n\to\infty}(\widetilde{S}_np)(x)=p(x),
	\,\, p\in\Pi.
	\]
 Therefore,
	Theorem~\ref{T16} yields
	\[
	\lim_{n\to\infty}(\widetilde{S}_nf)(x)=f(x),\,\, x\geq0,
	\]
	for every $f$ in the domain of $\widetilde{S}_n$ such that $|f(x)|\leq p(x)$ for some $p\in\Pi$.
\end{example}

\begin{example}
	Consider the Post--Widder operator fixing $e_0$ and $e_j$, $j\in\mathbb N$
	(see \cite{Dolomites}),
	\[
	(\widetilde{P}_{n,j}f)(x)
	=
	\frac{j n^n}{x^{nj}(n-1)!}
	\int_0^\infty
	u^{nj-1}e^{-nu^j/x^j}f(u)\,du.
	\]
	Its moments are
	\[
	\left(\widetilde{P}_{n,j}e_r\right)(x)
	=
	\frac{\Gamma(n+r/j)}{n^{r/j}\Gamma(n)}x^r.
	\]
	Using the asymptotic relation
	\begin{equation}\label{eq.x1}
	\frac{\Gamma(n+c)}{\Gamma(n)}\sim n^c,
	\qquad n\to\infty,
\end{equation}
	we obtain
	\[
	\lim_{n\to\infty}
	\left(\widetilde{P}_{n,j}e_r\right)(x)
	=x^r=e_r(x).
	\]
	Thus, the convergence assumptions of Theorem~\ref{T16} are satisfied.
	
\end{example}

\begin{example}
	Consider the Gamma operator fixing $e_0$ and $e_j$, $j\in\mathbb N$
	(see \cite{Dolomites}),
	\[
	(\widetilde{G}_{n,j}f)(x)
	=
	\frac{(n x^j)^{n+1}}{n!}\,j
	\int_0^\infty
	\frac{1}{u^{(n+1)j+1}}
	e^{-n x^j/u^j}f(u)\,du.
	\]
 For $r\in\mathbb N_0$, with $r<j(n+1)$, the  moments are given by
	\[
	\left(\widetilde{G}_{n,j}e_r\right)(x)
	=
	\frac{\Gamma\left(n+1-r/j\right)}
	{n^{1-r/j}\Gamma(n)}x^r.
	\]
	Using (\ref{eq.x1})
		with $c=1-r/j$, we get
	\[
	\lim_{n\to\infty}
	\left(\widetilde{G}_{n,j}e_r\right)(x)
	=x^r=e_r(x).
	\]
	Hence, Theorem~\ref{T16} applies.
\end{example}

\begin{example}
	Consider the Sz\'{a}sz--Mirakjan--Durrmeyer type operators fixing
	$e_0$ and $e_j$, $j\in\mathbb{N}$ (see \cite[Definition~1]{Abel}).
	If $f\in C[0,\infty)$ satisfies
	\[
	|f(x)|=O\left(e^{Ax}\right) \textrm{ as } x\to\infty,
	\]
	for some $A\geq 0$, then $S_{n,j}$ is defined, for $n>A$, by
	\[
	(S_{n,j}f)(x)
	=
	f(0)\sum_{k=0}^{j-1}s_{n,k}(x)
	+
	\sum_{k=j}^{\infty}s_{n,k}(x)n
	\int_{0}^{\infty}s_{n,k-j}(t)f(t)\,dt.
	\]
	The moments of the operators $S_{n,j}$ are given by
	$S_{n,j}e_0=e_0$ and, for $r\in\mathbb{N}$, by
	\[
	(S_{n,j}e_r)(x)
	=
	\sum_{k=0}^{r}
	\frac{k!}{n^k}
	\binom{r}{k}
	\binom{r-j}{k}
	x^{r-k}
	-
	\frac{r!}{n^r}e^{-nx}
	\sum_{k=0}^{j-1-r}
	\binom{k+r-j}{r}
	\frac{(nx)^k}{k!}.
	\]
	The second sum disappears if $r\geq j$
	(see \cite[Proposition~4]{Abel}).
		It follows that, for every fixed $r\in\mathbb N$,
	\[
	\lim_{n\to\infty}(S_{n,j}e_r)(x)
	=x^r=e_r(x),\qquad x\geq0.
	\]
	Therefore, the hypotheses of Theorem~\ref{T16} are fulfilled for the
	operators $S_{n,j}$.
		
\end{example}

\begin{example}
	Other Post--Widder type operators fixing $e_0$ and $e_j$, and Gamma
	type operators fixing $e_0$ and $e_j$, are studied in \cite{Dolomites}.
	They are defined by
$$	\left(P_{n,j}f\right)(x)
	:=
	P_n\left(
	f\left(\frac{\cdot}{\sqrt[j]{\lambda_{n,j}}}\right)
	\right)(x),$$
		and
		$$\left(G_{n,j}f\right)(x)
	:=
	G_n\left(
	f\left(\frac{\cdot}{\sqrt[j]{\mu_{n,j}}}\right)
	\right)(x),$$
	where
	\[
	\lambda_{n,j}:=\frac{n^{\overline{j}}}{n^j},
	\,\,
	\mu_{n,j}:=\frac{n^j}{n^{\underline{j}}},
	\,\, n\geq j,
	\]
	and $n^{\overline{j}}$ and $n^{\underline{j}}$ denote the rising and
	falling factorials, respectively.
	
	The moments of these operators are given by
	\[
	P_{n,j}e_r
	=
	\frac{\lambda_{n,r}}{\lambda_{n,j}^{\,r/j}}e_r,
	\,\,
	G_{n,j}e_r
	=
	\frac{\mu_{n,r}}{\mu_{n,j}^{\,r/j}}e_r,\,\, n\geq r.
	\]
	Since, for every fixed $r,j\in\mathbb N$,
$
	\lambda_{n,r}\to1,\, \mu_{n,r}\to1,
	\, n\to\infty,
	$
	it follows that
	\[
	P_{n,j}e_r\to e_r,
	\,\,
	G_{n,j}e_r\to e_r.
	\]
	Hence, Theorem~\ref{T16} applies to both families of operators.
\end{example}

\section{Convergence for functions with exponential growth}

Let $\exp_{\lambda}(x)=e^{\lambda x}$, $\lambda,x\in{\mathbb{R}}$. Consider
also the function $\alpha(x)=\cosh(x)-1$, $x\in {\mathbb{R}}$.

\begin{theorem}
\label{T5.8} Let $D$ be a linear subspace of $C[0,\infty )$ such that $\exp
_{\lambda }\in D$, $\lambda \in \mathbb{R}$. Let $L_{n}:D\rightarrow
C[0,\infty )$ be positive linear operators, $L_{n}\exp _{\lambda
}\rightarrow \exp _{\lambda }$ as $n\rightarrow \infty $ pointwise on $%
[0,\infty )$, $\lambda \in \mathbb{R}$. Let $f\in D$, $%
\left
\vert f(x)\right \vert \leq A\exp _{\omega }(x),\,x\geq 0$ for some $%
A>0$ and $\omega >0$. Then $L_{n}f\rightarrow f$ 
pointwise on $[0,\infty )$, as $n\rightarrow \infty $.
\end{theorem}

For the proof we need the following lemma.

\begin{lemma} \label{l3.2}
Let $A>0$, $\omega>0$, $f \in C[0,\infty)$, such that 
\begin{equation*}
A \exp_{\omega}(x) \ge f(x), \quad x \ge 0.
\end{equation*}
Fix $a \ge 0$, $\varepsilon > 0$. Then there exist $\ell>0$, $K \in \mathbb{N%
}$, $m \in \mathbb{N}$ such that 
\begin{equation}
K\, \alpha^{m}\! \left(\frac{x-a}{\ell}\right) + f(a) + \varepsilon \ge
f(x), \quad x \ge 0.  \label{ew1}
\end{equation}
\end{lemma}

\begin{proof}
Let $a>0$. There exists $\ell\in (0,a/2]$  such that 
\begin{equation}  \label{ew2}
f(a)+\varepsilon \ge f(x), \quad x \in (a-2\ell, a+2\ell).
\end{equation}
Let $m \in \mathbb{N}$, $m > \omega \ell$. Then 
\begin{equation*}
\lim_{x \to \infty} \frac{\alpha^{m}\! \left(\frac{x-a}{\ell}\right) +
f(a)+\varepsilon}{A \exp_{\omega}(x)} = \infty.
\end{equation*}
So, there exists $M>a+2\ell$ such that 
\begin{equation}  \label{ew3}
\alpha^{m}\! \left(\frac{x-a}{\ell}\right) + f(a)+\varepsilon \ge f(x),
\quad x > M.
\end{equation}
Set 
\begin{equation*}
U := [0, a-2\ell] \cup [a+2\ell, M].
\end{equation*}
Since
$
\alpha\!\left(\frac{x-a}{\ell}\right)>0,
\,\, x\in U,
$
we can choose $K\in\mathbb{N}$ such that
\begin{equation}  \label{ew4}
\min \left \{ K\, \alpha^{m}\! \left(\frac{x-a}{\ell}\right) +
f(a)+\varepsilon \; \middle|\; x \in U \right \} \ge \max \left \{ A
\exp_{\omega}(x) \; \middle|\; x \in U \right \}.
\end{equation}
Then (\ref{ew1}) is a consequence of (\ref{ew2})--(\ref{ew4}). The proof is
similar for $a = 0$.
\end{proof}

\medskip

\begin{proof}[Proof of Theorem~\protect \ref{T5.8}]
Fix $a\ge 0$ and $\varepsilon>0$.

\noindent Since 
\begin{equation*}
f(x)\le A\exp_\omega(x),\qquad x\ge 0,
\end{equation*}
Lemma \ref{l3.2} yields a function 
\begin{equation*}
g(x)=K\alpha^m\! \left(\frac{x-a}{\ell}\right)+f(a)+\varepsilon
\end{equation*}
such that $g(x)\ge f(x)$ for all $x\ge 0$.

\noindent Since $\alpha(x)=\cosh(x)-1$,  the function $g$ is a finite linear
combination of exponentials, hence $g\in D$. By positivity, 
\begin{equation*}
(L_n f)(a)\le (L_n g)(a).
\end{equation*}
Using the assumption $\displaystyle \lim_{n\to \infty}(L_n\exp_\lambda)(a)=
\exp_\lambda(a)$ for every $\lambda \in \mathbb{R}$, we obtain 
\begin{equation*}
\displaystyle \lim_{n\to \infty}(L_ng)(a)= g(a)=f(a)+\varepsilon,
\end{equation*}
and therefore 
\begin{equation*}
\limsup_{n\to \infty}(L_nf)(a)\le f(a)+\varepsilon.
\end{equation*}
Applying Lemma \ref{l3.2} to $-f$, we similarly obtain 
\begin{equation*}
\liminf_{n\to \infty}(L_nf)(a)\ge f(a)-\varepsilon.
\end{equation*}
Hence 
\begin{equation*}
f(a)-\varepsilon \le \liminf_{n\to \infty}(L_nf)(a) \le \limsup_{n\to
\infty}(L_nf)(a) \le f(a)+\varepsilon.
\end{equation*}
Since $\epsilon>0$ can be chosen to be arbitrary small, we conclude that 
\begin{equation*}
\displaystyle \lim_{n\to \infty}(L_nf)(a)= f(a).
\end{equation*}
Since $a\ge0$ was arbitrary, the convergence is pointwise on $[0,\infty)$.
\end{proof}

\noindent Denote
	\[
C_{\exp}[0,\infty)
:=
\left\{
f\in C[0,\infty):
|f(t)|\leq A \exp_{\omega}(t),\ t\geq 0,
\text{ for some } A>0,\ \omega>0
\right\}.
\]

\begin{example}
	Let $f\in C_{\exp}[0,\infty)$. Then the Sz\'{a}sz--Mirakjan operators
	$S_n$ satisfy
	$
	S_nf\rightarrow f
	$
	as $n\rightarrow\infty$, pointwise on $[0,\infty)$.
	
	Indeed, for all $\lambda\in\mathbb{R}$ and $x\geq0$, we have
	\[
	(S_n\exp_{\lambda})(x)
	=
	e^{nx(e^{\lambda/n}-1)}
	\to
	\exp_{\lambda}(x),
	\qquad n\to\infty.
	\]
	Thus, the assertion follows from Theorem~\ref{T5.8}.
\end{example}

\begin{example}
	Consider the positive linear operators
	\[
	(A_n^0f)(x)
	=
	\frac{1}{I_0(nx)}
	\sum_{k=0}^{\infty}
	\frac{\left(\frac{nx}{2}\right)^{2k}}
	{\left(k!\right)^2}
	f\left(\frac{2k}{n}\right),
	\,\, x\geq 0,\,\, n\in\mathbb{N},\,\, f\in C_{\exp}[0,\infty),
	\]
	introduced in \cite{Herzog1} and studied in \cite{Herzog2}, where
	\[
	I_0(z)
	=
	\sum_{k=0}^{\infty}
	\frac{\left(\frac{z}{2}\right)^{2k}}
	{\left(k!\right)^2}
	\]
	is the modified Bessel function of the first kind of order zero.
	
	The operators $A_n^0$ are well defined on $C_{\exp}[0,\infty)$.
	Indeed, if $f\in C_{\exp}[0,\infty)$ with
	\[
	|f(t)|\leq A e^{\omega t},\qquad t\geq0,
	\]
	then
	\[
	\begin{aligned}
	(A_n^0|f|)(x)
		&\leq
		\frac{A}{I_0(nx)}
		\sum_{k=0}^{\infty}
		\frac{\left(\frac{nx}{2}\right)^{2k}}
		{\left(\Gamma(k+1)\right)^2}
		e^{2k\omega/n}\\
		&=
		A\frac{I_0\left(nx e^{\omega/n}\right)}{I_0(nx)}
		<\infty.
	\end{aligned}
	\]
	
	Now, let $\lambda\in\mathbb{R}$. We have
	\[
	\begin{aligned}
		(A_n^0\exp_\lambda)(x)
		&=
		\frac{1}{I_0(nx)}
		\sum_{k=0}^{\infty}
		\frac{\left(\frac{nx}{2}\right)^{2k}}
		{\left(\Gamma(k+1)\right)^2}
		e^{2k\lambda/n}\\
		&=
		\frac{I_0\left(nx e^{\lambda/n}\right)}{I_0(nx)}.
	\end{aligned}
	\]
	For $x>0$, using the asymptotic relation
	\[
	I_0(z)\sim \frac{e^z}{\sqrt{2\pi z}},
	\qquad z\to\infty,
	\]
	we obtain
	\[
	(A_n^0\exp_\lambda)(x)
	\sim
	\exp\left\{nx\left(e^{\lambda/n}-1\right)\right\}
	e^{-\lambda/(2n)}.
	\]
	Since
	\[
	n\left(e^{\lambda/n}-1\right)\to\lambda,
	\qquad n\to\infty,
	\]
	it follows that
	\[
	(A_n^0\exp_\lambda)(x)
	\to e^{\lambda x}
	=
	\exp_\lambda(x),\,\, n\to\infty,\,\,
	 x>0.
	\]
	For $x=0$, since $I_0(0)=1$, we have
	\[
	(A_n^0\exp_\lambda)(0)=1=\exp_\lambda(0).
	\]
	Consequently,
	\[
	A_n^0\exp_\lambda
	\to
	\exp_\lambda
	\qquad n\to\infty
	\]
	pointwise on $[0,\infty)$, for every $\lambda\in\mathbb{R}$.
	
	Therefore, by Theorem~\ref{T5.8}, for each
	$f\in C_{\exp}[0,\infty)$,
	\[
	(A_n^0f)(x)\to f(x),
	\,\, n\to\infty,
	\]
	for every $x\geq0$.
\end{example}

\begin{example} 		Consider the positive linear operators
	\[
	(R_nf)(x)
	:=
	\frac{1}{\cosh(nx)}
	\sum_{k=0}^{\infty}
	\frac{(nx)^{2k}}{(2k)!}
	f\left(\frac{2k}{n}\right),
	\,\, x\geq 0,\,\, n\in\mathbb{N},\,\, f\in C_{\exp}[0,\infty),
	\]
	introduced in \cite{LR}.
	
	The operators $R_n$ are well defined for continuous functions of
	exponential growth. Indeed, if
	\[
	|f(t)|\leq A e^{\omega t},
	\qquad t\geq0,
	\]
	for some $A>0$ and $\omega>0$, then
	\[
	\begin{aligned}
		(R_n|f|)(x)
		&\leq
		\frac{A}{\cosh(nx)}
		\sum_{k=0}^{\infty}
		\frac{(nx)^{2k}}{(2k)!}
		e^{2k\omega/n}\\
		&=
		A\frac{\cosh\!\left(nx e^{\omega/n}\right)}
		{\cosh(nx)}
		<\infty.
	\end{aligned}
	\]
	
	For $\lambda\in\mathbb{R}$, we have
	\[
	\begin{aligned}
		(R_n\exp_\lambda)(x)
		&=
		\frac{1}{\cosh(nx)}
		\sum_{k=0}^{\infty}
		\frac{(nx)^{2k}}{(2k)!}
		e^{2k\lambda/n}\\
		&=
		\frac{\cosh\!\left(nx e^{\lambda/n}\right)}
		{\cosh(nx)}.
	\end{aligned}
	\]
	For $x>0$, it follows that
	\[
	\begin{aligned}
		(R_n\exp_\lambda)(x)
		&=
		\exp\left\{nx\left(e^{\lambda/n}-1\right)\right\}
		\frac{1+e^{-2nx e^{\lambda/n}}}
		{1+e^{-2nx}}.
	\end{aligned}
	\]
	Since
	\[
	n\left(e^{\lambda/n}-1\right)\to\lambda\,\, \textrm{ and } \,\,	\frac{1+e^{-2nx e^{\lambda/n}}}
	{1+e^{-2nx}}
	\to 1,\,\,
	n\to\infty,
	\]
		we obtain
	\[
	(R_n\exp_\lambda)(x)
	\to e^{\lambda x}
	=
	\exp_\lambda(x),\,\,n\to\infty,\,\,
	 x>0.
	\]
	For $x=0$, we have
	\[
	(R_n\exp_\lambda)(0)=1=\exp_\lambda(0).
	\]
	Thus,
	\[
	R_n\exp_\lambda\to\exp_\lambda,\,\,  n\to\infty
	\]
	pointwise on $[0,\infty)$, for every $\lambda\in\mathbb{R}$.
	
	Therefore, by Theorem~\ref{T5.8}, if
	 $f\in C_{\exp}[0,\infty)$  with
	\[
	|f(x)|\leq A\exp_\omega(x),
	\,\, x\geq 0,
	\]
	for some $A>0$ and $\omega>0$, then
	\[
	(R_nf)(x)\to f(x),
	\,\, n\to\infty,
	\]
	for every $x\geq0$.
\end{example}

%%\bibliography{references}

\end{document}